\documentclass[11 pt]{amsart}
\usepackage[T1]{fontenc}
\usepackage{amsmath,amsthm,amsfonts,amssymb,latexsym, mathrsfs}
\usepackage{hyperref}
\usepackage{color}
\usepackage{xcolor}
\usepackage{blindtext}
\usepackage{tikz-cd}
\usepackage{stmaryrd}

\theoremstyle{plain}
\newtheorem{theorem}{Theorem}[section]
\numberwithin{equation}{section} 
\numberwithin{figure}{section} 
\theoremstyle{plain}

\theoremstyle{plain}

\theoremstyle{plain}

\theoremstyle{plain}
\newtheorem{corollary}[theorem]{Corollary}
\theoremstyle{plain}
\newtheorem{fact}[theorem]{Fact}
\theoremstyle{plain}
\newtheorem{example}[theorem]{Example}
\theoremstyle{plain}
\newtheorem{lemma}[theorem]{Lemma}
\theoremstyle{plain}

\theoremstyle{plain}

\theoremstyle{plain}

\theoremstyle{plain}
\newtheorem{remark}[theorem]{Remark}
\theoremstyle{plain}

\theoremstyle{plain}

\theoremstyle{definition}
\newtheorem{definition}[theorem]{Definition}
\theoremstyle{definition}

\usepackage{comment}

\providecommand{\dotdiv}{
  \mathbin{
    \vphantom{+}
    \text{
      \mathsurround=0pt 
      \ooalign{
        \noalign{\kern-.35ex}
        \hidewidth$\smash{\cdot}$\hidewidth\cr 
        \noalign{\kern.35ex}
        $-$\cr 
      }%
    }%
  }%
}

\def\Ind#1#2{#1\setbox0=\hbox{$#1x$}\kern\wd0\hbox to 0pt{\hss$#1\mid$\hss}
\lower.9\ht0\hbox to 0pt{\hss$#1\smile$\hss}\kern\wd0}
\def\Notind#1#2{#1\setbox0=\hbox{$#1x$}\kern\wd0\hbox to 0pt{\mathchardef\nn="0236\hss$#1\nn$\kern1.4\wd0\hss}\hbox to 0pt{\hss$#1\mid$\hss}\lower.9\ht0
\hbox to 0pt{\hss$#1\smile$\hss}\kern\wd0}

\def\tp{\mathop{tp}}
\def\dom{\mathop{dom}}

\newcommand{\mons}{\Bar{M}}
\newcommand{\C}{\mathbb{C}}

\newcommand{\F}{\mathfrak{F}}

\newcommand{\Lang}{\mathcal{L}}

\def\Th{\mathop{Th}}

\tikzset{
  symbol/.style={
    draw=none,
    every to/.append style={
      edge node={node [sloped, allow upside down, auto=false]{$#1$}}}
  }
}

\title{Definably Amenable Groups in Continuous Logic}

\begin{document}
\title{Definably Amenable Groups in Continuous Logic}

\author{Juan Felipe Carmona}
\address{Tecnológico de Monterrey, Campus Querétaro, Querétaro, México}
\email{jf.carmona@tec.mx}

\author{Alf Onshuus}
\address{Universidad de los Andes, Cra 1 No 18A-10, Bogotá, Colombia}
\email{aonshuus@uniandes.edu.co}

\begin{abstract}
    We introduce the notions of definable amenability and extreme definable amenability for groups in continuous structures and conduct an extensive analysis of them, drawing parallels with the classical first-order case. We characterize both notions using fixed-point properties. We show that stable and ultracompact groups are definably amenable and prove that, for groups definable in dependent theories, definable amenability is equivalent to the existence of an $f$-generic type. Finally, we show the randomizations of first-order definably amenable groups are extremely definably amenable.
\end{abstract}
\maketitle

\section{Introduction}
\medskip

\medskip

A topological group $G$ is \emph{amenable} if there is a left invariant mean $\mathfrak m$ on $RUCB(G)$, the set of right
uniformly continuous bounded functions on $G$. 

The group is \emph{extremely amenable} if every separately continuous action of $G$ on a compact Hausdorff space $X$ has a fixed point.

We would like to generalize these notions to the context of continuous logic. In this setting, a structure is a metric space and the predicates are uniformly continuous functions. A definable group is a definable subset of $M^n$ with a definable group structure. It is proved in \cite{ivanov2021metric} that a definable group has an equivalent metric bi-invariant with respect to the group operation (Thus the predicates are right-uniformly continuous). We therefore introduce the following definition:

\begin{definition}
    Let $M$ be a continuous structure and $G\subset M^n$ be a definable group.
    We say that $G$ is \emph{definably amenable} if there is a
    mean $\mathfrak m$ on the set of $M$-definable predicates, such that
    $\mathfrak m(\varphi)=\mathfrak m(g\cdot \varphi)$ for any $M$-definable predicate $\varphi$.
\end{definition}

If $M$ is a classical structure, then this definition coincides with the classical one. Namely, a group $G$ is definably amenable if there is a $G$-invariant Keisler probability measure on its definable subsets.

Now, it is proved in \cite{granirer1965extremely} that a group is extremely amenable if and only if it has a \emph{multiplicative} invariant mean on $RUCB(G)$. We generalize this notion to the continuous setting:

\begin{definition}
    
A definable group $G$ is \emph{extremely definably amenable} (EDA) if it has a $G$-invariant global type.

\end{definition}

We will prove (Theorem \ref{typeshom}) that types are multiplicative means, so this definition is consistent with the classical one.

 Although not with this name, these groups were studied in \cite{berenstein2021definable} in the context of randomizations of first order amenable dependent groups.

 In classical model theory, definably amenable groups generalize stable and o-minimal groups to the dependent setting, providing a well-behaved notion of generic types and helping to understand the structure of the quotient  $G/G^{00}$ , where  $G^{00}$  is the smallest type-definable subgroup of bounded index. Outside dependent theories, the class of definable groups is also interesting, as it contains pseudofinite groups.

 We show that definable amenability (as we define it here) play a similar role in the continuous case. We show that stable and pseudocompact groups are definably amenable (Theorem \ref{stable are amenable} and Theorem \ref{ultraproduct}), and we characterize definable amenability via the existence of fixed-point of certain $G$-definable actions (Theorem \ref{fixed point} ). We also show that, for dependent theories, definable amenability is equivalent to the existence of an $f$-generic type (this is, a type that does not fork over a $M_0$ nor its translates) (Theorem \ref{f-generic}). Finally, we show that the randomizations of classical definably amenable groups are extremely definably amenable (Theorem \ref{randomizations are extremely amenable}).

\section{Definable Amenability and Fixed-Point Properties}\label{Preliminaries}
\subsection{Preliminaries}

    From now on, let us fix a theory $T$ in a continuous language $\Lang$ and let $\mons$ be a $\kappa$-saturated, $\kappa$-strongly homogeneous model of $T$, for $\kappa$ large enough.  Recall that a \emph{type over \( A \)} is a functional 
    \[
    p: \mathcal{P}_A^{\bar{x}} \to \mathbb{R}
    \]
    such that there exists a model \( M \) of \( T \) and a tuple \( \bar{a} \in M^{|\bar{x}|} \) satisfying 
    \[
    p(\varphi) = \varphi(\bar{a})
    \]
    for all definable predicates \( \varphi \) over \( A \). 

    A \emph{partial type} over \( A \) is a partial function \( p: \mathcal{P}_A^{\bar{x}} \to \mathbb{R} \) that corresponds to the restriction of a complete type to a subset of definable predicates.

    The \emph{space of complete types} over \( A \) with free variables in \( \bar{x} \), denoted \( S_A^{\bar{x}} \), is the set of all complete types over \( A \). This space is equipped with the \emph{logic topology}, which is the weak\(^*\) topology induced from the dual of \( \mathcal{P}_A^{\bar{x}} \).

Equivalently, a partial type $\pi$ can be seen as the set of closed sets (\emph{closed conditions}) of the form 
    \[
    r_1\leq \varphi(\bar{x}) \leq r_2,
    \]
    such that $r_1\leq\pi(\varphi)\leq r_2$.

    A set of closed conditions determines a partial type if its intersection is non-empty (in the space of types). By compactness (Theorem \ref{compactness} below), this is equivalent to the set of closed conditions having the finite intersection property.

\begin{theorem}[Compactness \cite{yaacov2008model} Proposition 8.6]\label{compactness}
    The space of types $S_A^{\Bar{x}}$ is compact with the Logic Topology.
\end{theorem}
Since the type space is compact and distinct types are separated by predicates, the Stone-Weierstrass theorem guarantees the following:
\begin{theorem}[\cite{yaacov2008model} Theorem 9.9]
    Any continuous function $f : S^{\bar{x}}_A \to \mathbb{R}$ is of the form $f_\varphi(q)=q(\varphi)$ for some definable predicate $\varphi$.
\end{theorem}

We finish the preliminaries by characterizing types as multiplicative functionals:

\begin{theorem}\label{typeshom}
    A functional \(p:\mathcal{P}_A^{\Bar{x}}\to \mathbb{R}\) is a Banach-algebra homomorphism if and only if it is a complete type.
\end{theorem}

\begin{proof}
    (\(\Leftarrow\)) This direction is straightforward.

    (\(\Rightarrow\)) First notice that every closed condition of the form $r_1\leq \varphi(\Bar{x})\leq r_2$ is equivalent to a condition of the form
    $\psi(\Bar{x})= 0$, where 
    
    \[\psi(\bar x)=   (\varphi(\Bar{x}) \dotdiv r_2) + (r_1 \dotdiv \varphi(\Bar{x}))\]   

    Therefore, we need to show that the kernel of \(p\) is finitely consistent. First, note that if a predicate \(\varphi \) satisfies \(\varphi(x) > 0\) in every model \(M\), then \(\varphi(x) \geq \epsilon\) for some \(\epsilon > 0\). Consequently, \(p(\varphi) \geq \epsilon\). Therefore, if \(\varphi \geq 0\) and \(p(\varphi) = 0\), there exists \(a\) such that \(\varphi(a) = 0\).

    Now, take \(\varphi_1, \ldots, \varphi_n \in \ker(p)\). Since \(p\) is a homomorphism, we have:
    \[
        p\left(\sum_{i=1}^n \varphi_i^2\right) = \left(\sum_{i=1}^n p(\varphi_i)\right)^2 - 2 \sum_{i \neq j} p(\varphi_i) p(\varphi_j) = 0.
    \]

    Therefore, \(\sum_{i=1}^n \varphi_i^2\) has a realization \(a\). In particular, this implies \(\varphi_i(a) = 0\) for all \(i\).

\end{proof}

\subsection{Definable amenability and extreme definable amenability}

Analogous notions of amenability and extremely amenability for definable groups and definable theories in first-order logic have been studied widely in \cite{hrushovski2008groups} and \cite{hrushovski2019amenability}. In this paper, we propose such notions for the metric case.

Let us recall that a \emph{definable group} $(G,\cdot)$ in a structure $M$ is a group such that the set $G\subset M^n$ is definable and the multiplication is the restriction of a definable function $\cdot:M^n\times M^n\to M^n$ to $G$.

The predicate that defines the group is denoted by $G(x)$ and it induces a group in every model of $T$. This is, for every model $N\models T$, the set $G(N)$ is a group with the same operation as $G(M)$.
Notice that if $\varphi(\Bar{x})$ is a definable predicate, the predicate $g\cdot \varphi(x):=\varphi(g^{-1}\cdot x)$ is definable since the function $\cdot$ is definable (see Fact \ref{definable sets and functions}).

If $G$ is a definable group and $p$ is a complete type in $\Lang (M)$ such that $p(G(x))=0$,  then, for each $h\in G$, we define $q= h\cdot p$ to be the type $tp(h\cdot g/M)$, where $g$ is any realization of $p$.

The type $p$ is said to be \emph{$G$-invariant} if $h\cdot p=p$ for every $h\in G$.

\begin{definition}
    A functional $\mathfrak m$ on $\mathcal{P}_{A}^{\bar x}$ is called a \emph{mean over $A$} if $\mathfrak m(\varphi)\geq 0$ when $\varphi\geq 0$ and $\mathfrak m(\textbf{1})=1$, where $\textbf{1}$ is the constant function of value $1$.
\end{definition}

Notice that any type is a mean, so we can consider the space of types as a subspace of the space of means.

\begin{definition}
    \begin{enumerate}
        \item A mean over $A$  is $A$-invariant if it is invariant under the automorphisms that fix $A$ pointwise.

        \item A mean over $A$ is \emph{A-definable} if it is A-invariant and for every $A$-predicate $\varphi(\bar x, \bar y)$, and $r, s \in \mathbb{R}$, the set \[\{q \in S_A^{\bar y} : r<\mu(\varphi(\bar x; \bar b)) < s \mbox{ for any }b \in \mons; b\models q\}\] is
              an open subset of $S_A^{\bar y}$.
    \end{enumerate}
\end{definition}

\begin{remark}
    The previous definition implies that the function $f_{\varphi}(y)\to \mu(\varphi(\bar x, y))$ is continuous for every predicate $\varphi(\bar x, y)$, hence, there is a predicate $P_{\varphi}(y)$ definable over $M$ such that $P_{\varphi}(y)=\mu(\varphi(\bar x, y))$. This coincides with the definition of definability in the first order case.
\end{remark}

The discussion in the introduction motivates the following definition:

\begin{definition}
    Let $M$ be a continuous structure and $G\subset M^n$ be a definable group.
    We say that $G$ is \emph{definably amenable} if there is a
    mean $\mathfrak m$ on the set of $M$-definable predicates $\mathcal{P}_{M}^{\bar x}$, such that
    $\mathfrak m(\varphi)=\mathfrak m(g\cdot \varphi)$ for any $\varphi\in \mathcal{P}_{M}^{\bar x}$.

    The group $G$ is \emph{extremely definably amenable} (EDA) if there is a $G$-invariant complete $M$-type.
\end{definition}

We next characterize definable amenability in terms of the existence of a measure in the space of types, for this, we need the Riesz representation theorem (See Theorem \ref{Riesz}).

\begin{theorem}\label{measure}
    The group $G$ is definably amenable if and only if there exists a finitely additive probability measure on the space of types $S_G(M)$ that is invariant under the action of $G$.
\end{theorem}

\begin{proof}
    The space of types \( X = S_G(M) \) is compact and Hausdorff, and by the previous remark, \( C(X) \) is the space of predicates, so, by the R–M–K Representation Theorem \ref{Riesz} a mean is given by integration with respect to a probability measure; since it is a Radon measure, the set $C(X)$ is dense in $L_1(X)$, in particular, for every measurable set $A$, there exists a sequence of definable predicates $\varphi_n$ such that $\varphi_n\to \chi_A$ in the $L_1$ norm. Therefore, $\mu(A)=\lim \mathfrak m(\varphi_n) = \lim \mathfrak m(g\cdot \varphi_n) = \mu(gA)$, for every $g\in G$. this implies  that the \( G \)-invariance of the mean is equivalent to the \( G \)-invariance of the measure.
\end{proof}

Quoting known results of continuous logic, we can now proof some facts of our notions of amenability and extreme amenability. This provides evidence that they are the correct ones in this context.

It is noted in \cite{ivanov2021metric} that every group $G$ that admits a continuous structure admits a bi-invariant metric equivalent to the original metric. Thus, predicates restricted to $G$ are in $RUCB(G)$.

\begin{remark}
    If a metric group $G$ is (extremely) amenable, then it is (extremely) definably amenable
\end{remark}

\begin{proof}
    If $G$ is amenable, then there is a mean on $RUCB(G)$. Since the predicates are all bounded right uniformly continuous, we have that $G$ is definably amenable.

    If $G$ is extremely amenable, then its action on the space of types $S_G(M)$ has a fixed point, this is an invariant type, hence $G$ is definably extremely amenable.
\end{proof}

\begin{theorem}\label{ultraproduct}
    Let $\{M_i\}_{i\in I}$ be a family of metric structures let $G_i(M_i)$ be a definable group. If $G_i(M_i)$ is (extremely)  definably amenable for every $i$, then $G(\prod_{\mathcal{U}} M_i)$ is (extremely) definably amenable.
\end{theorem}

\begin{proof}
    Let $\mathfrak m_i$ be an invariant mean on $G_i(M_i)$ and $\varphi(x,a_i)$ a predicate with parameters in $a_i$. By definition of the signature $L$, there exists a real number $r$ such that $M_i\models |\varphi(x,y)|\leq r$. In particular, $M_i \models |\varphi(x,a_i)|\leq r$; thus, \[\{\mathfrak m_i(\varphi(x,a_i)): i\in I\}\subset [-r,r].\]

    Since $[-r,r]$ is compact and Hausdorff, the ultralimit $\lim_{i,\mathcal{U}}(\mathfrak m_i(\varphi(x,a_i)))$ exists. We define $\hat{\mathfrak m}$ in the ultraproduct as

    \[\hat{\mathfrak m}(\varphi([x],[a_i]))=\lim_{i,\mathcal{U}}(\mathfrak m_i(\varphi(x,a_i)))\]

    Now, it is routine to check that the ultralimit of positive functionals is a positive functional and that $\mathfrak m (\mathbf 1) = 1$, then $\mathfrak m$
    is a mean. Moreover, since $\mathfrak m_i$ is $G(M_i)$ invariant, we have that     $\mathfrak m$ is $G(\prod_{\mathcal U} M_i)$ invariant. Finally, if $\mathfrak m_i$
    is a type, by  Theorem \ref{typeshom} it is a Banach-homomorphism, then     $\mathfrak m$ is also a Banach-homomorphism, therefore it is a type.
\end{proof}

\begin{corollary}\label{elementary}
    If $M\prec N$ and $G$ is a definable group, then $G(M)$ is (extremely)  definably amenable if and only if $G(N)$ is.
\end{corollary}

\begin{proof}
    If $G(M)$ is (extremely) definably amenable, then by Keisler-Shelah theorem for continuous logic, there exists  $\prod_{\mathcal{U}} M$ an ultrapower of $M$ such that $N\prec \prod_{\mathcal{U}} M$. By the previous theorem, $G(\prod_{\mathcal{U}} M)$ is (extremely)  definably amenable.
\end{proof}

Recall that a group $G$ is called \emph{pseudocompact} if it is elementary equivalent to an ultraproduct of compact groups. Thus, we have that:

\begin{corollary}
    Pseudocompact groups are definably amenable.
\end{corollary}

\begin{example}
    Let $M_n(\C)$ be the full matrix C*-algebra over the complex numbers. Since its unitary group is compact, it is amenable (hence definably amenable). However, if $\mathcal{U}$ is a non-principal ultrafilter over $\mathbb{N}$, the ultraproduct $\prod_{\mathcal U} M_n(\C)$ is not nuclear (nor elementary equivalent to a nuclear algebra, see Proposition 7.2.5. \cite{farah2016model}). This implies that the unitary group of $\prod_{\mathcal U} M_n(\C)$ is not amenable with the relative Banach weak topology (Theorem \ref{paterson}) (in particular, it cannot be amenable with the metric topology). However, definable amenability is preserved under ultraproducts (Theorem \ref{ultraproduct}), and the unitary group of the ultraproduct is the ultraproduct of the unitary groups (Theorem \ref{Ge}). Thus, $U(\prod_{\mathcal U} M_n(\C))$ is definably amenable, even though it is not amenable.
\end{example}

\subsection{Fixed-point theorems}

We know investigate the topological properties of the space of means:

The space $V=\mathcal{P}_{M}^{\bar x}$ is a vector space. Thus we can consider the dual space $V^*$. with the weak$^*$ topology on $V^*$. It is well-known that this space is Hausdorff and locally convex (see, for example, \cite{rudin1991functional}).

By Alaouglu's theorem, the unit ball of $V$ is compact.

Let us denote by $\Sigma(M)$ the space of means over $\mathcal{P}_{M}^{\bar x}$ and, as usual, $S_M^{\bar x}$ denotes the space of complete types over $M$. Then,  we may assume that $S_M^{\bar x}\subset \Sigma(M)$.

$\Sigma(M)$ is closed in the unit ball: it is the preimage of $1$ of the (continuous) function $v'\to v'(\mathbf{1})$, whence it is closed as well. Thus, it is compact.

We next show that the closure of the convex hull of $S_M^{\bar x}$ is $\Sigma(M)$. The idea is that compact sets in locally convex spaces are the convex hulls of their extreme points.

Let us recall that a point \( p \) in a convex set \( S \) is called an \emph{extreme point} if \( p \) cannot be expressed as a non-trivial convex combination of other points in \( S \). Formally, \( p \in S \) is an extreme point if, whenever \( p = t x + (1 - t) y \) for some \( x, y \in S \) and \( t \in (0,1) \), it follows that \( x = p \) and \( y = p \).

\begin{theorem}\label{convexhull}
    The closure of the convex hull of the $S_M^{\bar x}$ is $\Sigma(M)$.
\end{theorem}

\begin{proof}
    
    By Theorem \ref{phelps}, where $A=(\mathcal{P}_{M}^{\bar x})$, $B$ the set of constant functions from $\mathbb{R}$ to $\mathbb{R}$ (so $B$ is isomorphic to $\mathbb{R}$ as an algebra), and $\Sigma(M)=K(A,B)$, we have that the extreme points of $\Sigma(M)$ are the multiplicative operators, which, by Remark \ref{typeshom}, are precisely the elements of $S_M^{\bar x}$ . 
    Since $\Sigma(M)$ is a compact convex subset of $V'=(\mathcal{P}_{M}^{\bar x})'$, we can apply the Krein-Milman Theorem (Theorem \ref{kreinmilman}), therefore, $\Sigma(M)$ is the closure of the convex hull of $S_M^{\bar x}$.
\end{proof}

Having established that the closure of the convex hull of types equates to the space of means, we now connect this characterization to the fixed-point property results. The compactness and convexity of this space ensure that definable actions of groups on compact sets admit fixed points.

\begin{definition}

    \begin{itemize}
        \item Let $D$ be a definable set in a structure $M$ and $K$ be a compact Hausdorff space. We say that $f:D\to K$ is \emph{definable} if, for every $C\subset U\subset K$, with $C$ closed and $U$ open, there exists a definable predicate $\varphi(x)$ and $\epsilon>0$ such that \[f^{-1}(C)\subset \{d\in D: \varphi(d)=0\} \subset \{d\in D: \varphi(d)<\epsilon\} \subset f^{-1}(U).\]

        \item A \emph{definable action} of $G$ on $X$ (or \emph{definable $G$-flow}) is an action $G\times K\to K$ such that
              \begin{itemize}
                  \item For every $g\in G$, $x\to gx$ is a homeomorphism.
                  \item For every $k\in K$, $g \to gk$ is definable.

              \end{itemize}
    \end{itemize}

\end{definition}

If all the types based on $G$ are definable, then the action of $G$ onto its space of types is definable. However, requiring all the types to be definable is too strong (for instance, a theory $T$ is stable if and only if all its types over all models are definable). Thus, we propose a weaker definition of definable action:

\begin{definition}

    A \emph{weak definable action} of $G$ on a compact space $K$ (or \emph{weak definable $G$-flow}) is an action $G\times K\to K$ such that
    \begin{itemize}
        \item For every $g\in G$, $x\to gx$ is an homeomorphism.
        \item For some $k_0\in K$, $g\to gk_0$ is definable.
    \end{itemize}
    In this case, the triple $(K,G,k_0)$ is called a \emph{weak definable ambit}.

\end{definition}

\begin{definition}
    Let $S_G(M)$ be the space of types over $M$ concentrating in $G$.
\end{definition}

\begin{theorem}
    $(S_G(M), G, tp(e/M))$, is a weak definable ambit.
\end{theorem}

\begin{proof}

\begin{itemize}
    \item For every $g\in G$, $p\to gp$ is a bijection in $S_G(M)$ that sends the open basic sets $[\varphi(x)<\epsilon]$ to the open set $[g^{-1}\varphi(x)<\epsilon]$. Thus, it is a homeomorphism.
    \item Let $C\subset U\subset S_G(M)$, with $C$ closed and $U$ open. By definition of the logic topology, there exists a definable predicate $\varphi(x)$ and $\epsilon>0$ such that \[C\subset [\varphi(x)=0] \subset [\varphi(x)<\epsilon] \subset U.\] By taking the inverse image by the function $f(g)= g\cdot tp(e/M)$ we get that
          \[f^{-1}(C)\subset \{g\in G : \varphi(g)=0\} \subset \{g\in G: \varphi(g)<\epsilon\} \subset f^{-1}(U).\]
\end{itemize}

\end{proof}

Moreover, by a similar reasoning it can be shown that the action of $G$ on its space of means makes $(\Sigma(G),G,tp(e))$ a weak definable ambit as well. If all the types (respectively means) are definable, then $(S_G(M),G,tp(e))$ ($(\Sigma(G),G,tp(e))$ respectively) is a definable ambit.

The following result characterizes the weak ambit $(S_G (M), G, tp(e))$ as universal in the category of weak definable ambits.

A similar result has been proved in  \cite{gismatullin2014compactifications} for the first-order case (when all the types are definable). 

\begin{theorem}
    The weak definable ambit $(S_G(M),G,tp(e))$ is universal in the category of weak definable ambits. 
\end{theorem}

\begin{proof}
    Let $(X,G,x_0)$ be a weak definable ambit and let $f:G\to X$ be given by $f(g)=g\cdot x_0$. We extend $f$ to a function $f^*:S_G(M)\to X$ as follows:
    \[f^*(p)= \bigcap_{\varphi \in ker(p)} \overline{f(\varphi(G))}\]

    First of all, let us check that the function $f^*$ is well-defined: suppose that $a\neq b$ are $\bigcap_{\varphi \in ker(p)} \overline{f(\varphi(G))}$ and let $C_1$ and $C_2$ be disjoint neighborhoods of $a$ and $b$. Since $f$ is definable, then \[f^{-1}(C_1)\subset \{g\in G : \varphi(g)=0\} \subset \{g\in G: \varphi(g)<\epsilon\} \subset f^{-1}(C_2^c).\] Thus, $b\notin \overline{f(\varphi(G))}$, a contradiction.

    Now let us check that $f^*$ is a homomorphism of $G$-ambits (by construction, it is the only possible one):
    \begin{itemize}

         \item The function $f^*$ is a homomorphism of $G$-spaces:
                      \begin{align*}
                          gf^*(p) & = g\bigcap_{\varphi \in ker(p)} \overline{f(\varphi(G))}  \\
                                  & = \bigcap_{ \varphi \in ker(p)} g\overline{f(\varphi(G))} \\
                                  & = \bigcap_{ \varphi \in ker(p)} g\overline{\varphi(G)x_0} \\
                                  & = \bigcap_{ \varphi \in ker(p)} \overline{g\varphi(G)x_0} 
                          && \text{(translations are homeomorphisms)}\\
                                  & = \bigcap_{ \varphi \in ker(p)} \overline{f (g\varphi(G))} \\
                                  & = \bigcap_{\psi \in ker(gp)} \overline{\psi(G)}     \\
                                  & = f^*(gp) .                              
                      \end{align*}

        \item The image of $tp(e)$ is $x_0$: since $f(e)=x_0$, we have that $x_0\in \overline{f(\varphi(G))}$ for every $\varphi \in ker(tp(e))$. Thus, we have
        \[
            x_0= \bigcap_{\varphi \in ker(tp(e))} \overline{f(\varphi(G))} = f^*(tp(e)).
        \]

    \end{itemize}

\end{proof}

In particular, if all the types are definable, then $(S_G(M),G,tp(e))$ is universal in the category of definable ambits.

Using Theorem \ref{convexhull}, we get the following corollary:

\begin{corollary}
    The weak definable ambit $(\Sigma(G),G,tp(e))$ is universal in the category of weak definable ambits $(K,G,k_0)$, where $K$ is a compact convex set of a locally convex vector space.
\end{corollary}

Finally, we can establish the main theorem of this section, whose proof follows from the discussion above.

\begin{theorem}\label{fixed point}
    Let $G$ be a definable group. Then

    \begin{enumerate}

        \item Characterization of extreme definable amenability:
              \begin{itemize}
                  \item $G$ is extremely definably amenable if and only if every weak definable $G$-action over a compact set has a fixed point.
              \end{itemize}

              Moreover, if all the types in $S_G(M)$ are definable, then:

              \begin{itemize}
                  \item $G$ is extremely definably amenable if and only if every definable action over a compact set has a fixed point.
              \end{itemize}

        \item Characterization of definable amenability:

              \begin{itemize}
                  \item $G$ is definably amenable if and only if every weak definable action over a compact convex subset on a locally convex space has a fixed point.
              \end{itemize}

              Moreover, if all the means in $\Sigma(G)$ are definable, then:

              \begin{itemize}
                  \item $G$ is definably amenable if and only if every definable action over a compact convex subset on a locally convex space has a fixed point.
              \end{itemize}

    \end{enumerate}

\end{theorem}

\section{Stable Groups}\label{Stable}

In this section, we show that stable groups are definably amenable. The proof  relies on several results proved by Ben-Yaacov in \cite{yaacov2010stability}. Let us recall the definitions and results we need:

\begin{definition}
    A predicate $\varphi(x,y)$ is \emph{stable} if for every $(a_i)_{i<\omega}$ and $(b_j)_{j<\omega}$, we have \[\lim_{i\to \infty}\lim_{j\to \infty}\varphi(a_i,b_j)=\lim_{j\to \infty}\lim_{i\to \infty}\varphi(a_i,b_j).\]
    (Assuming that one of the limits exists).
    A theory $T$ is \emph{stable} if all its predicates are stable in every model $M\models T$.

    A group $G$ is \emph{stable} if $G$ is definable in some stable theory $T$.

\end{definition}

If $G$ is a definable group, then $G$ acts continuously on $S_G(M)$: $gp(x)=p(g^{-1}x)$. This will help us define generic types in continuous logic:

\begin{definition}[\cite{yaacov2010stability}]
    Let $S$ be a $G$-space. A set $X\subset S$ is \emph{generic} if finitely many $G$-translates of $X$ cover $S$.
    A type $p$ is \emph{generic} if every open set of $S_G(M)$ containing $p$ is generic.
\end{definition}

\begin{example}
    The group $S_1$ is compact (hence stable). The formulas $|x|<\epsilon$ are clearly generics, but the set $|x|=0$ is not generic.
\end{example}

A group is \emph{definably connected} if $G=G^{00}$. 

\begin{theorem}
For every type $p$ we have that $stab(p)\subseteq G^{00}_M$. 
\end{theorem}

We adapt the proof from \cite{wagner2000simple}, Lemma 4.1.23 to the continuous setting:

\begin{proof}
    
The equivalence relation $\sim$ on $G$, defined by
\[
x \sim y \iff xy^{-1} \in G^{00}_M,
\]
is type-definable over $M$. Since $G^{00}_M$ is the smallest type-definable subgroup of bounded index, we can express $G^{00}_M$ as an intersection of zero-sets of definable predicates:
\[
    G^{00}_M = \bigcap_{i \in I} \{g \in G : \varphi_i(g) = 0\}.
\]
Thus, $x \sim y$ holds if and only if $\varphi_i(xy^{-1}) = 0$ for all $i \in I$, making $\sim$ a type-definable equivalence relation.

Furthermore, since $G^{00}_M$ has bounded index in $G$, the quotient $G/G^{00}_M$ has boundedly many classes, with the number of classes equal to the index of $G^{00}_M$ in $G$. Therefore, any type $p \in S_G(M)$ specifies in which $\sim$-class (i.e., which coset of $G^{00}_M$) its realizations lie. It follows that
\[
\mathrm{stab}(p) \subseteq G^{00}_M.
\]

\end{proof}

Thus, if $G$ is EDA, then it is definably connected.

\begin{theorem}[\cite{yaacov2010stability}]
    Let $G$ be a stable group, then:
    \begin{itemize}
        \item The set of generic types $Gen$ of $G$ is non-empty.
        \item The group $G$ has a connected component $G^{00}$, this is, there exists the smallest type-definable group of bounded index.
        \item The group $G$ acts in $Gen$ and this action is homeomorphic to the action on $G$ on $G/G^{00}$.
    \end{itemize}
\end{theorem}

The proof of the following fact is analogue to the classical case (see, for example, Lemma 8.10 in \cite{simon2015guide}).

\begin{lemma}
    The group $G/G^{00}$ is a compact group.
\end{lemma}

With these results, we can prove the main theorem of this section:

\begin{theorem}\label{stable are amenable}
    Stable groups are definably amenable.
\end{theorem}

\begin{proof}
    Since $G/G^{00}$ is a compact group, it has an invariant mean $m^*$ on all its bounded uniformly continuous functions.
    Let $\Psi: G/G^{00}\to Gen$ be the isomorphism (as $G$-sets), let us define a function $m$ on all the predicates $m(\varphi)$ as $m^*(\varphi^* \circ \Psi)$, where $\varphi^*:Gen\to \mathbb{R}$:  $\varphi^*(p)=p(\varphi)$. It is easy to check that $m^*$ is an invariant mean on all the predicates:
    \begin{itemize}
        \item $m(1)=m^*(1^* \circ \Psi)=m^*(1)=1$.
        \item $m(\varphi+\psi)=m^*((\varphi+\psi)^* \circ \Psi)=m^*(\varphi^* \circ \Psi + \psi^* \circ \Psi)=m^*(\varphi^* \circ \Psi)+m^*(\psi^* \circ \Psi)=m(\varphi)+m(\psi)$. (Because $m^*$ is a mean).
        \item $m(g\varphi)=m^*((g\varphi)^* \circ \Psi)=m^*(g^* (\varphi^* \circ \Psi))=m^*(\varphi^* \circ \Psi)=m(\varphi)$. (Because $m^*$ is $G$-invariant).
    \end{itemize}
\end{proof}

\begin{example}
    Let $(\Omega,\mu)$ be a probability space. We can define a metric group $(B,*,d)$, where $B$ is the set of measurable sets of $\Omega$, the operation group is given by $X*Y= X\triangle Y$ and the distance is given by $d(X,Y)=\mu(X\triangle Y)$.
    This is an $\omega$-stable group, hence definably amenable.
\end{example}

We call this group the \emph{probability group} of $(\Omega,\mu)$.

In fact, this group can be described as the \emph{randomization} of the group $\mathbb{Z}/2\mathbb{Z}$ (which is definably amenable since it is finite). By Theorem \ref{randomizations are extremely amenable}, we have that the probability group is extremely definably amenable.

\section{Generic Types, and S1 Ideals in dependent groups} \label{S1 f-generic}

Let $G$ be a definable group, by $G^{00}_A$ we denote the intersection of all
type-definable subgroups over $A$ of bounded index. If $G^{00}_A$ does not depend on $A$ we say that $G^{00}$ exists and it is equal to $G^{00}_A$.

The existence of $G^{00}$ in continuous dependent theories can be proved analogously to the classic case (see \cite{berenstein2021definable}, Theorem 2.19).

If the group is extremely definably amenable, we know that $G^{00}=G$, (so $G/G^{00}$ is trivial). We will show here that, for definably amenable groups $G/G^{00}$ is a compact group and $G^{00}$ is the stabilizer of a generic type (with respect to a forking notion).

Let us recall the definitions of forking and dividing in continuous logic:

\begin{definition}\cite{yaacov2008model}
    A type $\pi(x;B_0)$ \emph{divides} over $C$ if there exists an indiscernible sequence $\{B_i\}_{i<\omega}$ such that
    the $\bigcup \pi(x;B_i)$ is not satisfiable.
    A definable predicate $\varphi(x;b_0)$ \emph{divides} over $C$ if the type $\pi=\{\varphi(x;b_0)=0\}$ divides over $C$.
    We say that $\pi(x,B_0)$ \emph{forks} over $C$ if there exists $D\supset B_0\cup C$ such that every complete type $p$ over $D$, with $p\supset \pi(x,B_0)$, divides over $C$.
\end{definition}

\begin{remark}\label{open divides}
    If a definable predicate $\varphi(x,b_0)$ divides over $C$, then there exists $\epsilon > 0$ and an indiscernible sequence $\{b_i\}_{i<\omega}$ such that $|\varphi(x,b_i)|<\epsilon$ is not satisfiable
\end{remark}

\begin{proof}
    
Since the predicate divides, there exists an indiscernible sequence $\{b_i\}_{i<\omega}$  such that the set \{$|\varphi(x,b_i)|\leq 1/n: i<\omega, n\in \mathbb{N}$\} is not satisfiable. Thus, by compactness, there exists $N$ such that \{$|\varphi(x,b_i)|\leq 1/N: i<\omega,$\} is not satisfiable. Taking $\epsilon=1/N$ we get the desired result.
\end{proof}

Now let use define dependent theories:

\begin{definition}\cite{BenYaacov2009}
    A definable predicate \( \varphi(x,y) \) has the \emph{independence property} (IP) if there are \( r < s \) such that for all \( m \) there exist \( (\bar{b}_n : n < m) \) and \( (\bar{a}_w : w \subseteq m) \) satisfying:
    \[
    \varphi(\bar{a}_w, \bar{b}_n) \leq r \iff n \in w,
    \]
    \[
    \varphi(\bar{a}_w, \bar{b}_n) \geq s \iff n \notin w.
    \]

    A theory has the \emph{independence property} if there is a predicate with the independence property. Otherwise we say that the theory is \emph{dependent}.
\end{definition}

\begin{theorem}
    Let $T$ be a dependent theory. A global type $p\in S(M)$ is $A$-invariant if and only if $p$ does not fork over $A$.
\end{theorem}

The proof is analogous to the classic one \cite{simon2015guide}.

\begin{proof}
    Since $p$ is a global type, not forking over $A$ is equivalent to not dividing over $A$. Hence every $A$-invariant type does not fork over $A$ (this direction does not require dependent).
    Now, assume that $p$ is not $A$-invariant. Hence \[r=  p(\varphi(x,a_0))\neq  p(\varphi(x,a_1))=s\] for some $A$-indiscernible sequence $(a_i)_{i<\omega}$. By $dependence$, there does not exist any element $b$ such that  $\varphi(b,a_{2n})=r$ for and $\varphi(b,a_{2n+1})=s$, whence, the predicate $(\varphi(x,a_{0})-r)(\varphi(x,a_{1})-s)$ in $p$ divides.
\end{proof}

 We will also need the notion of $f$-generic types, defined as in \cite{hrushovski2008groups}:

\begin{definition}
    A global type $p$ is \emph{$f$-generic} if there is a small model $M_0$ such that neither $p$ nor its translates (under the group action) fork over $M_0$.
\end{definition}

The following fact is well known in the first-order setting, but the proof can be adapted almost verbatim to the continuous setting:

\begin{fact}[Corollary 8.20 in \cite{simon2015guide}]
    If $G$ admits an $f$-generic type over some model $M$, then it admits an $f$-generic type over any model $M_0$.
\end{fact}

\subsection{Stabilizers and amenability in dependent theories}

In this section, we will assume that $G$ is definable in a structure $M$ and $Th(M)$ is dependent. 

Given a type $p$ over $M$ we define
\[stab(p):=\{h\in G \text{ such that } hp=p.\}\]

By saturation, if $h\in stab(p)$ then $\exists a,b\in M$ such that $a,b\models p|_{M_0}$ and $ha=b$.

\begin{definition}
    Let $A$ be any set of parameters. The group of \emph{Lascar strong automorphisms} of $\mons$ over $A$ is the group generated by all $\operatorname{Aut}(\mons/M)$ where the $M$ are models containing $A$. Two tuples $a$ and $b$ have the same \emph{Lascar strong type} over $A$ if $\alpha(a) = b$ for some $\alpha \in \operatorname{Aut}_f(\mons/A)$. We denote this by $\operatorname{Lstp}(a/A) = \operatorname{Lstp}(b/A)$.
    \end{definition}

The classical version of the following theorem can be found in \cite{simon2015guide}. Chapter 8.

\begin{theorem}\label{fgenericimplies}
    If $G$ is a definable group with an $f$-generic type $p$, then:
    \begin{enumerate}
        \item $stab(p)=G^{00}$.
        \item $G$ is definably amenable.
    \end{enumerate}
\end{theorem}

\begin{proof}
    (1) Let $p$ be the $f$-generic type over $M_0$.

    \noindent\textbf{Claim 1:} The group $stab(p)$ is a type-definable group. More precisely, \[stab(p)=\{g_1^{-1}g_2| \tp(g_1/M_0)=\tp(g_2/M_0)\}:\]

    \begin{proof}[Proof of Claim 1] Since $gp$ is $Aut_{M_0}(M)$ invariant (by f-genericity), then, for $g_1\equiv_{M_0}g_2$, we have that $g_1p=f(g_1p)=f(g_1)p=g_2p$, for any $f\in Aut_{M_0}(M)$. Whence $g_1^{-1}g_2p=p$.

        On the other hand, if $h\in stab(p)$, then $ha=b$ for some realizations of $p$ outside $M$. This implies that $ha'=b'$ for some realizations of $p|M_0$ inside $M$. By naming $g_1= b^{-1}$ and $g_2=a^{-1}$ we get the desired result. Thus, we conclude that $stab(p)$ coincides with the set of elements preserving type-invariance.

        To see that the set on the right is type-definable, notice that it is equal to 
        \[\bigcap_{\psi_i(y,m)}[\inf_{g_1, g_2 \in G} \big( |x - g_1 g_2^{-1}| + |\psi_i(g_1, m) - \psi_i(g_2, m)| \big)=0]\]
    \end{proof}

    Now, $stab(p)$ is a type-definable group and its index is bounded by \[|S_{M_0}(G)|<2^{|M_0|}<\Bar{\kappa}.\] This implies that \[G^{00}\subset stab(p).\]

    On the other hand, since $G^{00}$ is an $M_0$-invariant group of bounded index (because it is $\emptyset$-invariant), then $g_1\equiv_{G^{00}} g_2$ is an $M_0$-invariant bounded equivalence relation. Therefore, if \[Lstp(g_1/M_0)= tp(g_1/M_0) = tp(g_2/M_0) = Lstp(g_2/M_0)\], then by definition of the Lascar Type, we have that $g_1\equiv_R g_2$ for any bounded $M_0$-invariant equivalence relation $R$. In particular, $g_1^{-1}g_2\in G^{00}$. Thus, \[stab(p)=\{g_1^{-1}g_2| tp(g_1/M_0)=tp(g_2/M_0)\}\subset G^{00}.\] Whence \[stab(p)= G^{00}.\]

    (2) In order to prove that $G$ is definably amenable, let us assume first that $T$ is a countable theory over a countable language and let $\varphi:M\to \mathbb{R}$ be any predicate defined over a model $M_0$ of countable density character. Assume that $p$ is $f$-generic over $M_0$ as well.

    Let us define $f_{\varphi}:G/G^{00}\to \mathbb{R}$ as follows:
    \[ f_{\varphi}([g]) = p (g\varphi).\]

    Notice that $f_{\varphi}$ is well-defined since $p$ is $G^{00}$ invariant.

    \noindent\textbf{Claim 2:} $f_{\varphi}$ is measurable:

    \begin{proof}[Proof of Claim 2]

        We want to check that $\{[g]\in G/G^{00}: r<p(g\varphi(x))<s\}$ is Borel in $G/G^{00}$. It suffices to show that $\{g\in G: p(g\varphi(x))\in (r,s)\}$ is a countable union of closed sets of $G$ (because the projection of $S_G(M)$ onto $G/G^{00}$ is closed).

        If $(a_i)_{i<\omega}\models p^{\omega}_{M_0}$, then $\lim_{i\to \infty} \psi(a_i)=p(\psi(x))$. Thus, whe have that $p(\psi(x))\in (r,s)$ if and only if

        (*) There exists a rational number $q\in (r,s)$ such that, for \[\epsilon_q= \frac{1}{2}\min\{|q-r|,|q-s|\}\] there exists $N<\omega$ such that $|\psi(a_N)-q|<\epsilon_q$ and  $|\psi(a_i)-\psi(a_j)|<\epsilon_q$ for $i,j>N$.

        Fix an element $g\in G$. By dependence, for every $\epsilon$ there is a maximal $N$ such that $(a_i)_{i<N}\models p^{N}_{M_0}$ and $|g\varphi(a_i)-g\varphi(a_{i+1})|\geq \epsilon$ for $i<N$.


        Let
        \[ \Phi^q_{N}(g) = \Big\{\inf\limits_{x_1,...,x_N}|\varphi_i(x_1,...,x_n)|\cdot \left[\prod_{i<N} \left(\epsilon_q \dotdiv |g\varphi(x_i)-g\varphi(x_{i+1})|\right)\right] \cdot\]

        \[|g\varphi(x_N)-q| \dotdiv \epsilon_q : \varphi_i \in p \Big\}.\]

        (This type says that there are $N$ realizations of the type $p^{N}_{M_0}$ satisfying (*)).

        Hence,

        \[f_{\varphi}(g)= p(g\varphi)\in (r,s) \Longleftrightarrow g\in \bigcup_{q\in \mathbb{Q}\cap(r,s)} (\bigcup_{N<\omega} ( \Phi^q_{N}(g)\cap \neg (\Phi^q_{N+1}(g))).\]

        Finally, since $L$ is countable and $M_0$ has a countable dense set, we may assume that $(\Phi^q_{N+1}(g))$ is countable. Thus,  $\{g\in G: p(g\varphi(x))\in (r,s)\}$ is a countable union of closed sets and $\{[g]\in G/G^{00}: r<p(g\varphi(x))<s\}$ is Borel in $G/G^{00}$. Therefore $f_{\varphi}$ is measurable.

    \end{proof}

    \noindent\textbf{Claim 3:} $f_{g_0\varphi}= g_0\cdot f_{\varphi}$.

    \begin{proof}[Proof of Claim 3]
        $f_{g_0\varphi}([g])= p(g g_0 \varphi) = f_{\varphi}([g] g_0) = g_0\cdot f_{\varphi}([g])$.
    \end{proof}

    \noindent\textbf{Claim 4:} $f_{\varphi+\psi}= f_{\varphi}+f_{\psi}$.

    \begin{proof}[Proof of Claim 4]
        $f_{\varphi+\psi}[g]= p(g(\varphi+\psi))= p(g\varphi+g\psi)$

        $= p(g(\varphi)+p(g\psi))=(f_{\varphi}+f_{\psi})[g]$.
    \end{proof}

    Let us define $\mathfrak m$ as \[\mathfrak m(\varphi)= \int f_{\varphi}.\]

    \noindent\textbf{Claim 5:} $\mathfrak m$ is a $G$-invariant functional $\mathfrak m:Form\to\mathbb{R}$ and $\mathfrak m(1)=1$.

    \begin{proof}[Proof of Claim 5]
        \begin{itemize}
            \item $\mathfrak m$ is $G$-invariant: $\mathfrak m(g\varphi)=\int f_{g\varphi} = \int g\cdot f_{\varphi} = \int f_{\varphi}$.

            \item $\mathfrak m$ is a functional: $\mathfrak m(\varphi + \psi) = \int (f_{\varphi+\psi})= \int f_{\varphi}+f_{\psi} =\int f_{\varphi}+\int f_{\psi}$.

            \item $\mathfrak m(1)=1$: $f_{1}[g]= p(g\cdot 1)= p(1) = 1$, hence $\mathfrak m(1)=\int f_{1} = \int 1 =1$.
        \end{itemize}

    \end{proof}

    So if $G$ has an $f$-generic type, then it is definably amenable.\end{proof}

    Now, let us prove that the converse of the previous theorem is also true:
    
    \begin{theorem}\label{f-generic}
        
        $G$ is definably amenable, if and only if admits an $f$-generic type.

    \end{theorem}

    \begin{proof}
    Assume that $G$ is definably amenable with an $M_0$-invariant measure, but does not have any $f$-generic type over $M_0$, then every type contains a condition that divides over $M_0$. By Lemma \ref{open divides}, every type contains an open condition that divides over $M_0$.Then, by compactness, there is a finite subcover of open sets of the form $\varphi_i(x)<\epsilon_i$, $i\leq n$. But, since these conditions fork, the open sets have measure $0$. This contradicts the fact that the measure of the whole space is $1$.
    \end{proof}

    \subsection{S1-ideals} 

    We assume throughout this subsection that $G$ is a definably amenable group with mean $\mathfrak m$ and all the predicates live in $G$.  
    We will first show the existence of an $S1$ ideal of \emph{small positive predicates}. We then show that, if a type contains only predicates outside the ideal, then it is $f$-generic. 

\begin{definition} [Classical First-Order]

    Let $I$ be an ideal of the boolean Algebra of $\mons$-definable sets. We say that $I$ is $S1$  if for any formula $\varphi(x, y)$ and indiscernible $(a_i)_{i\in \mathbb{N}}$, if $\varphi(x, a_i) \cap \varphi(x, a_j) \in I$ for $i\neq j$, then some $\varphi(x, a_i) \in I$.

\end{definition}

In the classical first-order setting, it can be shown that if $G$ is definably amenable, then the set of small formulas is $S1$. In this context, by a small formula we mean a formula of measure 0.

We aim to establish an analogue of this fact in the continuous setting, drawing inspiration from measure theory. In this context, we need to refine the notion of ``small''. Intuitively, just as a set is considered ``small'' if it has measure zero, we propose that a predicate \( \varphi(x) \) should be considered ``small'' if the set of zeros of \( \varphi(x) \) is ``small'' in a similar sense.

To formalize this, assume \( \varphi(x) \geq 0 \). Since \( \varphi(x) > 0 \) for almost all \( x \) (in the sense of measure theory), we can find, for every $\epsilon$ a sufficiently large \(\lambda\) such that the measure of \( \lambda \varphi(x) < 1 \) is less than $\epsilon$ . This implies that the function 
\[
\mathfrak{m}(1 \dotdiv \lambda \varphi(x)) \to 0 \quad \text{as} \quad \lambda \to \infty,
\].

Since we do not have an explicit measure defined on subsets of the model, we adopt this behavior as the operational definition of ``small'' for \( \varphi(x) \). This approach mirrors the way measure theory defines smallness via limiting behavior.

From now on, we will be working almost exclusively with positive predicates with parameters.

\begin{definition}
    Let $\varphi$ and $\psi$ be positive predicates. We define the pointwise operations:
    \[
        (\varphi \wedge \psi) (x) := \max\{\varphi(x), \psi(x)\}, \quad (\varphi \vee \psi) (x) := \min\{\varphi(x), \psi(x)\}.
    \]
    The truncation of $\varphi$ is given by $\overline{\varphi} := 1 \vee \varphi$. Additionally, let
    \[
    \mathfrak{m}_{\infty}(\varphi) = \lim\limits_{\lambda \to \infty} \mathfrak{m}(\overline{\lambda \varphi}).
    \]
    We say that $\varphi$ is \emph{small} if $\mathfrak{m}_{\infty}(\varphi) = 1$.
\end{definition}

\begin{remark}
    The notation $\wedge$ and $\vee$ in this paper follows a convention that naturally arises from considering the opposite partial order on the set of positive predicates. 

    This choice ensures that these operations correspond to the \emph{lattice join} and \emph{lattice meet}, respectively, when considering the opposite order where $\varphi \preceq \psi$ means that $\varphi(x) \geq \psi(x)$ for all $x$. In this sense, our definitions align with standard lattice-theoretic conventions.

    Simultaneously, these operations coincide with the classical logical operations in first-order logic.
    
\end{remark}

\begin{remark}\label{bounded properties}
    If $\varphi$ and $\psi$ are positive predicates, then 
    \[
    \overline{\varphi} \wedge \overline{\psi} = \overline{\varphi \wedge \psi}, 
    \quad \overline{\varphi} \vee \overline{\psi} = \overline{\varphi \vee \psi}.
    \]
\end{remark}

\begin{lemma}
    If $\varphi$ and $\psi$ are small, then $\varphi \vee  \psi$ is small.
\end{lemma}

\begin{proof}
    Note that for any \(\lambda > 0\),
    \[
    \overline{\lambda (\varphi \vee  \psi)} = \overline{\lambda \varphi} \vee  \overline{\lambda \psi}.
    \]
    Since \(\mathfrak{m}\) is monotone and
    \[
    \mathfrak{m}(\overline{\lambda \varphi} \vee  \overline{\lambda \psi}) \geq \mathfrak{m}(\overline{\lambda \varphi}) \vee  \mathfrak{m}(\overline{\lambda \psi}),
    \]
    taking limits yields
    \[
    \mathfrak{m}_\infty(\varphi \vee  \psi) = 1.
    \]
\end{proof}

\begin{lemma}\label{two-predicate}
    For every $\varphi$ and $\psi$ positive predicates, we have 
    \[
    \mathfrak{m}_\infty(A\wedge B) + \mathfrak{m}_\infty(A\vee B) = \mathfrak{m}_\infty(A) + \mathfrak{m}_\infty(B).
    \]
\end{lemma}

\begin{proof}
    For any \(a,b\in [0,1]\) we have the identity
    \[
    a\wedge b + a\vee b = a+b.
    \]

    Applying this pointwise to $\varphi$ and $\psi$ yields 
    \[
    \overline{\varphi}\wedge \overline{\psi} + \overline{\varphi}\vee \overline{\psi} = \overline{\varphi}+\overline{\psi}.
    \]
    Thus, for every $\lambda > 0$,
    \[
    \mathfrak{m}(\overline{\lambda \varphi}\wedge \overline{\lambda \psi}) + \mathfrak{m}(\overline{\lambda \varphi}\vee \overline{\lambda \psi}) = \mathfrak{m}(\overline{\lambda \varphi}) + \mathfrak{m}(\overline{\lambda \psi}).
    \]
    Taking limits yields the desired result.

\end{proof}

\begin{definition}
    Let $I$ be a set of positive predicates. We say that $I$ is an \emph{ideal} if:

    \begin{enumerate}
        \item Whenever $\varphi\in I$ and $\psi\geq \varphi$, then $\psi\in I$.
        \item If $\varphi$ and $\psi$ are in $I$, then $\varphi \vee  \psi$ is in $I$.
        \item $1 \in I$.
    \end{enumerate}

    The ideal is \emph{invariant over $A$} (\emph{$A$-invariant}) if for every $\varphi(\Bar{x},\Bar{a})\in I$ and $\Bar{a}\equiv_A\Bar{a}'$, we have that $\varphi(\Bar{x},\Bar{a}')\in I$.

    We say that an $A$-invariant ideal $I$ is \emph{$S1$ over $A$} if, for every $A$-indiscernible sequence of definable predicates $(\varphi_i)_{i<\omega}$, if $\varphi_i \wedge  \varphi_j \in I$ for every ($i\neq j$), then $\varphi_i\in I$ for every $i$.
\end{definition}

\begin{theorem}
    If $G$ is definably amenable and the measure is $M$-invariant, then the set $I$ of small positive predicates is an $S1$-ideal over $M$.
\end{theorem}

\begin{proof}
    Let $(\varphi_i)_{i<\omega}$ be an $M$-indiscernible sequence of definable predicates such that $\varphi_i\wedge \varphi_j \in I$ for every $i\neq j$. We want to show that $\varphi_i\in I$ for every $i$.
    Assume for contradiction that none of the \( \varphi_i \) is small. By indiscernibility, there exists \(\delta>0\) (uniform for all \(i\)) such that
    \[
    \mathfrak{m}_\infty(\varphi_i)\le 1-\delta \quad\text{for every } i.
    \]

    We now prove by induction on \(n\) that for any distinct indices \(i_1,\dots,i_n\),
    \[
    \mathfrak{m}_\infty\Bigl(\varphi_{i_1}\vee \varphi_{i_2}\vee\cdots\vee \varphi_{i_n}\Bigr) \le 1 - n\delta.
    \]
    For \(n=1\) the statement is trivial since \(\mathfrak{m}_\infty(\varphi_{i_1})\le 1-\delta\) by assumption.
    
    Now, assume the claim holds for \(n=k\); that is,
    \[
    \mathfrak{m}_\infty\Bigl(\bigvee_{j=1}^k \varphi_{i_j}\Bigr) \le 1 - k\delta.
    \]
    We now consider \(n=k+1\). Apply the two–predicate result to the predicates
    \[
    A:=\bigvee_{j=1}^k \varphi_{i_j} \quad\text{and}\quad B:=\varphi_{i_{k+1}}.
    \]
    Since
    \[
    A\wedge B = \bigvee_{j=1}^{k} (\varphi_{i_j} \wedge \varphi_{i_{k+1}} ) \in I,
    \]
    we have
    \[
    \mathfrak{m}_\infty(A\wedge B)=1.
    \]
    Then by Lemma \ref{two-predicate}:
    
    \[
    \mathfrak{m}_\infty(A\vee B) = \mathfrak{m}_\infty(A) + \mathfrak{m}_\infty(B) - \mathfrak{m}_\infty(A\wedge B) = 1 - k\delta + 1 - \delta - 1 = 1 - (k+1)\delta.
    \]

    Thus, for $k$ big enough, we have that $\mathfrak{m}_\infty(\varphi_{i_1}\vee \varphi_{i_2}\vee\cdots\vee \varphi_{i_n})< 0$, which is a contradiction. 

\end{proof}

\begin{definition}\label{wide}
    A type $p$ is $I$-\emph{wide} if $ker(p)\cap I=\emptyset$.

\end{definition}

We can now relate all to $f$-generic types.

\begin{lemma}\label{positive implies small}
    If $I$ is an ideal and $\varphi(x,a)>0$ in $G$, then $\varphi\in I$.
\end{lemma}

\begin{proof}
    By saturation, $\varphi(x,a)\geq \lambda$ for some $\lambda>0$, thus, the predicate $\varphi$ is in $I$.
\end{proof}

\begin{theorem}
    If $I$ is a $S1$-ideal over $M$ and $p$ is $I$-wide, then it is $f$-generic.
\end{theorem}

\begin{proof}
    Clearly, if $p$ is $I$-wide, all its conjugates are, so it suffices to show that $p$ does not fork (divide) over $M$: otherwise, there exists a predicate $\varphi(x,a_0)\in ker(p)$ and an indiscernible sequence $(a_i)_{i<\omega}$ such that $\{\varphi(x,a_i)\}_{i<\omega}$ is $k$-inconsistent for some $k$. This implies that $\bigwedge_{i=0}^{k-1} |\varphi(x,a_i)|>0$, thus, by Lemma \ref{positive implies small}, it is small. Therefore, by the $S1$ property, the formula $|\varphi(x,a_0)|$ is in $I$, this is a contradiction.
\end{proof}

\begin{corollary}\label{final in section}
    If $G$ is definably amenable with an $M$-invariant mean and $I$ is the ideal of small formulas. Then any $I$-wide global type is $f$-generic.
\end{corollary}

\section{Randomizations}\label{Randomization}
In this section, we prove that the randomization or first-order definably amenable groups are extremely definably amenable. This generalizes the results from Berenstein and Mu\~noz \cite{berenstein2021definable} that was restricted to dependent groups. In particular, this gives us new examples of extremely definably amenable groups.

Let us recall briefly the construction of randomizations. For a more detailed exposition, we refer the reader to \cite{ben2009randomizations}:

Fix $(\Omega, B, \mu)$ a probability space and $L$ a first-order language.

We define a new language $L^R$ as $L^R=\{K,B,\top,\bot, \sqcap, \sqcup, \neg ,\llbracket \varphi_i(\bar x)\rrbracket\},$
where $K$ and $B$ are sorts, $\top,\bot, \sqcap, \sqcup, \neq$ are symbols for the boolean operations and $\llbracket \varphi_i(\bar x)\rrbracket$ are functions \[\llbracket \varphi_i(\bar x)\rrbracket: K^n\to B,\] where $n$ is the number of free variables of $\varphi_i$.

Given a first order $L$-structure $M$,  let $(\Omega, \mathcal B, \mu)$ be an atomless finitely additive probability space and $\mathcal K$ a set of functions \[\{f_i:\Omega\to M| i\in I\}.\]

Both $K$ and $B$ are pre-metric spaces with \[d_K(X,Y)=\mu(\{\omega \in \Omega| X(m)\neq Y(m)\})\] and \[d_B(A,B)=\mu(A\triangle B).\]

The pair $( \mathcal B,K)$, is called a \emph{randomization} of $M$ if, as a premetric structure in $L^R$, satisfies the following:

\begin{enumerate}
    \item For each $L$-formula $\varphi(\bar x)$ and  $f$ in $K^n$, we have
          \[\llbracket \varphi_i(X)\rrbracket= \{\omega \in \Omega: M\models \varphi(X(m))\}\]
    \item  For every $B \in \mathcal B$ and $\epsilon >0 $ there are $f, g \in \mathcal K$ such that \[\mu(\llbracket f = g\rrbracket \triangle B) < \epsilon\].
    \item For every $L$-formula  $\varphi(\bar x; y)$, $g$ in $\mathcal K^n$ and $\epsilon > 0$ there is $f \in \mathcal K$ such that \[\mu \llbracket \varphi(f, g)\rrbracket \triangle \llbracket \exists \varphi(\bar x,g)\rrbracket <\epsilon \]

          The interpretation of $\top,\bot, \sqcap, \sqcup, \neq$ are the usual boolean operations on $B$.

\end{enumerate}

The completion of $K$ and $B$ are continuous structures.

One associates an $L^R$ continuous theory $T^R$ to every first order $L$ theory $T$, see \cite{ben2009randomizations}. This theory $T^R$ is the common theory of all the randomizations of models of $T$.

A randomization \((K, \mathcal{B})\) of \(\mathcal{M}\) is \emph{full} if 
\begin{itemize}
    \item \(\mathcal{B}\) is equal to the set of all events \(\llbracket \psi(\bar{f}) \rrbracket\), where \(\psi(\bar{x})\) is a formula of \(L\) and \(\bar{f}\) is a tuple in \(K\).
    \item \(K\) is \emph{full in} \(\mathcal{M}^\Omega\), that is, for each formula \(\theta(x, y)\) of \(L\) and tuple \(\bar{g}\) in \(K\), there exists \(f \in K\) such that
    \[
    \llbracket \theta(f, \bar{g}) \rrbracket = \llbracket (\exists x \, \theta)(\bar{g}) \rrbracket.
    \]
\end{itemize}

 By a \emph{reduced pre-structure}, we will mean a pre-structure such that \(d_K\) and \(d_B\) are metrics. Then every pre-structure \((K, \mathcal{B})\) induces a unique reduced pre-structure \((\bar{K}, \bar{\mathcal{B}})\) by identifying elements which are at distance zero from each other. The induced continuous structure is then obtained by completing the metrics, and will be denoted by \((\hat{K}, \hat{\mathcal{B}})\). We say that a pre-structure \((K, \mathcal{B})\) is \emph{pre-complete} if the reduced pre-structure \((\bar{K}, \bar{\mathcal{B}})\) is already a continuous structure, that is, \((\hat{K}, \hat{\mathcal{B}}) = (\bar{K}, \bar{\mathcal{B}})\).

\begin{fact}[\cite{ben2009randomizations},]
    $T$ is complete/stable/dependent if and only if $T^R$ is.
\end{fact}

\begin{definition}
    Let $(\Omega,\mathcal B, \mu)$ be an atomless finitely additive probability space and let $M$ be a $L$-structure.

    A $\mathcal B$-simple random element of $M$ is a $\mathcal B$-measurable function in $M$ with finite range.

    A $\mathcal B$-countable random element of $M$ is a $\mathcal B$-measurable function in $M^{\Omega}$ with countable range.
\end{definition}

It is proven in \cite{keisler1999randomizing} that the set $K_S$ of $B$-simple random elements of $M$ is a full-randomization of $M$ and that $K_C$, the set of countable random elements is a full-randomization of $M$ that is pre complete. Moreover  $K_S$ is dense in $K_C$ (\cite{ben2009randomizations}).

Now we prove the main theorem of this section:

\begin{theorem}\label{randomizations are extremely amenable}
    Let $G$ be a definably amenable group with measure $\nu$. Then $G^R$ is extremely definably amenable.
\end{theorem}

\begin{proof}
    Since EDA is a property of the theory of $G^R$, we may assume that $M^R=K_C$ where $K_C$ is induced by the countable random elements. Moreover, without loss of generality, we may assume that $\Omega = [0,1]$.
    Let $p$ be the following type:
    \[p=\{\mu \llbracket \varphi(\bar x,\bar a) \rrbracket = \sum_{\bar m_i\in ran(\bar a)} \nu(\varphi(\bar x,\bar m_i))\mu(\bar a ^{-1}(\bar m_i)) | \Bar{a} \in K_S\}\]

    By quantifier elimination and density of $K_S$, this determines a unique complete type.

    Notice that the definition of $p$ is quite natural: if $\bar a$ has constant value $\bar m$, then $p(\llbracket \varphi(\bar x,\bar a) \rrbracket) = \nu(\varphi(\bar x,\bar m))$. The extension to non-constant variables $\bar a$ is basically the integral over that function (which is equal to a finite sum since $\bar a$ is simple.

    Let us check that $p$ is consistent: take $p_0$ a finite subset of $p$,

    \[p_0=\{\mu \llbracket \varphi_j(\bar x,\bar a_j) \rrbracket = \sum_{m_{i,j}\in ran(a_j)} \nu(\varphi_j(\bar x,\bar m_{i,j})\mu(\bar a_j ^{-1}(\bar m_{i,j})): j\leq n\}\]

    Let us define $A_{i,j}= \bar a_j ^{-1}(\bar m_{i,j})$. Notice that $\{A_{i,j}\}_{\bar m_{i,j}\in ran(a_j) }$ is a partition of $\Omega$. Hence, the non-empty intersections of the form $\bigcap_{j\leq n} A_{i_j} $ form a partition of $\Omega$.  We build $f$, a realization of $p_0$, as follows: For every \[I= \{(i_1,1),...,(i_n,n)\}. \] We name any of such intersections as  \[A_I= A_{i_1,1}\cap...\cap A_{i_n,n}. \]

    Take $B_1,...,B_{l}$ the atoms of the Boolean algebra generated by $\varphi_j(\bar x,\bar m_{i_j,j})$ and choose any partition of $A_I$ of the form $\{A_{I,r}\}_{r\leq l}$ such that $\mu (A_{I,r})= \nu(B_r)\mu(A_I)$, (it exists since $\sum \nu(B_r)=1$).

    Now, for each $A_{I,r}$, we take a point $b_{I,r}$ satisfying $B_r$ and define $f:\Omega:\to G$ such that $f[A_{I,r}]= b_{I,r}$. Since $\{A_{I,r}\}$ is a measurable partition of $\Omega$, then $f\in K_S$.

    Let us check that $f\models p_0$:

    \begin{align*}
        \mu \llbracket \varphi_j(f,\bar a_j) \rrbracket & = \mu \{\omega: \models \varphi_j(f(\omega),\bar a_j(\omega))\}                                                              \\
                                                        & = \sum_{\bar m_{i,j}\in ran(\bar a_j)} \mu \{\omega\in A_{i,j}: \models \varphi_j(f(\omega),m_{i_j,j})\}                            \\
                                                        & = \sum_{\bar m_{i,j}\in ran(\bar a_j)} \sum_{I\ni (i,j)} \mu \{\omega\in A_{I}: \models \varphi_j(f(\omega),\bar m_{i_j,j})\}            \\
                                                        & = \sum_{\bar m_{i,j}\in ran(\bar a_j)} \sum_{I\ni (i,j)} \sum_{r} \mu \{\omega\in A_{I,r}: \models \varphi_j(f(\omega),\bar m_{i_j,j})\} \\
                                                        & = \sum_{\bar m_{i,j}\in ran(\bar a_j)} \sum_{I\ni (i,j)} \sum_{r} \mu \{\omega\in A_{I,r}: \models \varphi_j(b_{I,r},\bar m_{i_j,j})\}   \\
                                                        & = \sum_{\bar m_{i,j}\in ran(\bar a_j)} \sum_{I\ni (i,j)} \sum_{B_r\subset \varphi_j(\bar x,m_{i_j,j}) } \mu  (A_{I,r})                   \\
                                                        & = \sum_{\bar m_{i,j}\in ran(\bar a_j)} \sum_{I\ni (i,j)} \mu(A_I) \sum_{B_r\subset \varphi_j(\bar x,m_{i_j,j}) } \nu(B_r)                \\
                                                        & = \sum_{\bar m_{i,j}\in ran(\bar a_j)} \mu(A_{i,j}) \nu(\varphi_j(\bar x,m_{i_j,j}))                                                     \\
                                                        & = \sum_{\bar m_{i,j}\in ran(\bar a_j)} \nu(\varphi_j(\bar x,\bar m_{i_j,j})\mu(\bar a_j ^{-1}(\bar m_{i_j,j}))
    \end{align*}

    Now, we check that $p$ is $G^R$-invariant. Again, since $K_S$ is dense, it is enough to prove that $p$ is invariant under the simple random elements in $G^R$:

    Let $g_R$ a simple random element in $G^R$ and let $h_R=g_R^{-1}$, then, in $p$, we have that
    \begin{align*}
        \mu \llbracket g_R\varphi(\bar x,\bar a) \rrbracket & = \mu \llbracket \varphi(h_R \bar x,\bar a) \rrbracket                                                                          \\
                                                       & = \sum_{h_i \in ran(h_R), m_j\in ran(\bar a)} \nu(\varphi(h_i \bar x,m_{j})(\mu(h_R ^{-1}(h_i))+\mu(\bar a ^{-1} (m_{j}))       \\
                                                       & = \sum_{ m_j\in ran(\bar a)} \nu(\varphi(h_i \bar x,m_{j}))\sum_{h_i\in ran(h_R)}(\mu(h_R ^{-1}(h_i))+\mu(\bar a ^{-1} (m_{j})) \\
    \end{align*}

    Finally, by amenability of $G$ we have that \[\nu(\varphi(h_i \bar x,\bar m_{j}))=\nu(\varphi(\bar x,\bar m_{j}))\]

    and, since $\mu(\Omega)=1$

    \[\sum_{h_i\in ran(h_R)}(\mu(h_R ^{-1}(h_i)))=1\]

    So we get that:

    \[ \mu \llbracket g_R\varphi(\bar x,\bar a) \rrbracket = \sum_{\bar m_i\in ran(\bar a)} \nu(\varphi(\bar x,\bar m_{i})\mu(\bar a ^{-1}(\bar m_{i}))\]

    Which is what we wanted to prove.
\end{proof}

\begin{corollary}
    The randomization of pseudofinite groups are extremely definably amenable.
\end{corollary}

\appendix

\section{Continuous Logic}

In this section we summarize the notation we will use throughout the paper and recall some key theorems from continuous logic. For a more detailed exposition the reader may check \cite{yaacov2008model}.

Note: in this paper we do not assume that truth values are restricted to the interval $[0,1]$; instead, every formula may have its own bounded range within the reals. This distinction is inconsequential since we can always normalize the formulas and their semantic remains intact, but it allows us to treat the set of formulas and predicates as a vector space, which facilitates the translation between logic and functional analysis (see \cite{farah2016model}).

A language $L$ is a triplet $(\mathcal S, F, R)$ which contains the following data:

\begin{enumerate}
    \item $\mathcal S$ is the set of sorts; for each sort $S \in \mathcal S$, there is a symbol $d_S$ meant to be interpreted as a metric bounded by a positive number $M_S$.

    \item $\mathcal F$ is the set of function symbols; for each function symbol $f\in  F$, we formally specify $\dom(f)$ as a sequence $(S_1, . . . , S_n)$ from $\mathcal S$ and $rng(f) \subset S$ for some
          $S\in  \mathcal S$. We will want $f$ to be interpreted as a uniformly continuous function.
          To this effect we additionally specify, as part of the language, functions $\delta^f_i: \mathbb R^+ \to \mathbb R^+$, for $i \leq n$. These functions are called uniform continuity moduli.

          (When interpreted, the function $f$ satisfies that \[d_{S_i}(x_i,y_i)<\delta^f_i(\epsilon) \mbox { for $i\leq n$} \implies d_S(f (\bar x),f(\bar y))<\epsilon.)\]

    \item $\mathcal R$ is the set of relation symbols; for each relation symbol $R \in \mathcal R$ we formally specify $\dom(R)$ as a sequence $(S_1, . . . , S_n)$ from $S$ and $rng(R) \subset K_R$ for some bounded interval $K_R$ in $\mathbb R$. As with function symbols, we additionally specify,
          as part of the language, functions $\delta^R_i:
              \mathbb R^+ \to \mathbb R^+$, for $i \leq n$, called uniform
          continuity moduli.
\end{enumerate}
For each sort $S \in \mathcal S$, we have infinitely many variables $x^S_i$.

Terms are defined inductively in the same way as in the first order logic:
\begin{enumerate}
    \item A variable $x^S_i$ is a term with domain and range $S$
    \item If $f \in \mathcal F$, \[dom(f) = (S_1, . . . , S_n)\] and
          \[\tau_1, . . . , \tau_n\] are terms with $rng(\tau_i) \subset S_i$ then
          $f(\tau_1, . . . , \tau_n$) is a term with range the same as $f$ and domain
          determined by the $\tau_i$'s.
\end{enumerate}

Formulas are also defined inductively, but the definition of connectives and quantifiers require more attention:

\begin{enumerate}
    \item If $R$ is a relation symbol (possibly a metric symbol) with domain $(S_1, . . . , S_n)$ and $\tau_1, . . . , \tau_n$ are terms with ranges $S_1, . . . , S_n$
          respectively then $R(\tau_1, . . . , \tau_n)$ is a formula. Both the domain and uniform
          continuity moduli of $R(\tau_1, . . . , \tau_n)$ can be determined naturally from $R$ and
          $\tau_1, . . . , \tau_n$. These are the atomic formulas.
    \item (Connectives) If $f : \mathbb R^n \to \mathbb R$ is a uniformly continuous function and $\varphi _1, . . . , \varphi_n$ are formulas then $f(\varphi_1, . . . , \varphi_n)$ is a formula. The domain and uniform continuity moduli are determined from $f$ and $\varphi_1, . . . , \varphi_n$.
    \item (Quantifiers) If $\varphi$ is a formula and $x$ is a variable of sort $S$, then both $\inf _{x\in S} \varphi$ and $\sup_{x\in S} \varphi$ are formulas.
\end{enumerate}

Given a language $L=(\mathcal S, \mathcal F, \mathcal R)$, a metric structure will be a collection of metric spaces, one for each sort $S$, with a metric $d^M_S$ of diameter at most $M_S$, together with a collection of functions $F^M_i$ and relations $R^M_i$ interpreted in the natural way.

The set of formulas with free variables in $\Bar{x}$ in the language $\Lang$ is denoted by $\F_{\Lang}^{\Bar{x}}$. The set of sentences $\F^{()}_{\Lang}$ is denoted by $Sent_{\Lang}$.

When interpreted in a metric structure $M$, formulas become bounded uniformly continuous functions $\varphi^M(\bar x): M^n\to \mathbb R$. If the formula $\varphi$ is a sentence, then $\varphi^M$ is just a real number.

A \emph{theory} $T$ is a set of sentences. The \emph{theory} of a metric structure $M$ is the set of sentences $\varphi$ such that $\varphi^M=0$.  We say that $M\models T$ if $T\subset \Th(M)$.

Given a theory $T$, we may endow $\F_{\Lang}^{\Bar{x}}$ with the seminorm \[|\varphi(\Bar{x})|:= \sup\{|\varphi(\Bar{a})|: \Bar{a}\in M^n, M\models T\}.\]

The quotient of $\F_{\Lang}^{\Bar{x}}$ by the formulas of seminorm 0 is a normed vector space whose completion is denoted as $\mathcal{P}^{\Bar{x}}_T$, the set of \emph{definable predicates} (without parameters). In other words, a function $\Phi(\Bar{x})$ is a definable predicate if there is a sequence of formulas $\varphi(\Bar{x})$ that converges uniformly to $\Phi(\Bar{x})$ in every model $M\models T$.

We will denote as $\mathcal{P}_{T(A)}^{\Bar{x}}$ the set of predicates with parameters in the set $A$. If the theory $T$ is clear from the context, we simply write $\mathcal{P}_A^{\Bar{x}}$.

\begin{definition}
    A closed set $D$ of $M^n$ is a {\it definable set} if the relation $dist(\Bar{x},D)$ is a definable predicate.
    A function $M^n\to M^k$ is a {\it definable function} if the relation $d(f(\Bar{x}),\bar y)$ is a definable predicate.
\end{definition}

The main facts about definable sets and definable functions we will use are the following (see \cite{yaacov2008model}):

\begin{fact}\label{definable sets and functions}
\begin{itemize}
    \item Quantifying a definable predicate over a definable set gives a definable predicate: if $D$ is a definable set and $\varphi(x,y)$ is a definable predicate, then $\sup_{x\in D}\varphi(x,y)$ and $\inf_{x\in D}\varphi(x,y)$ are definable predicates.
    \item The substitution of a definable function in a definable predicate gives a definable predicate.
    \item Definable sets are stable under ultraproducts: if $D$ is a definable set, then there exists a definable predicate $\varphi(x)$ such that $D$ is the zero-set of $\varphi(x)$ and for every ultraproduct $M_{\mathcal U}=\prod_{\mathcal{U}} M_i $ of models of $Th(M)$, we have that $\varphi(M_{\mathcal{U}})=\prod_{\mathcal{U}} \varphi(M_i)$.
\end{itemize}
\end{fact}

\section{Functional Analysis}

In this section, we recall some basic definitions and results from functional analysis that we will use in the paper.

A \emph{Radon measure} on a topological space \(X\) is a measure \(\mu\) defined on the Borel \(\sigma\)-algebra of \(X\) that satisfies the following properties:
\begin{enumerate}
    \item \textbf{Locally finite:} For every point \(x \in X\), there exists a neighborhood \(U\) of \(x\) such that \(\mu(U) < \infty\).
    \item \textbf{Inner regularity:} For every Borel set \(A \subseteq X\), \(\mu(A)\) is the supremum of \(\mu(K)\) over all compact subsets \(K \subseteq A\).
    \item \textbf{Outer regularity:} For every Borel set \(A \subseteq X\), \(\mu(A)\) is the infimum of \(\mu(U)\) over all open sets \(U \supseteq A\).
\end{enumerate}

\begin{theorem}[Riesz–Markov–Kakutani Representation Theorem]\label{Riesz}
    Let \( X \) be a compact Hausdorff space. The dual space of \( C(X) \), the space of continuous real-valued functions on \( X \) with the sup norm, is isometrically isomorphic to the space of Radon measures on \( X \) with the total variation norm. 
    
    That is, for every bounded linear functional \( L: C(X) \to \mathbb{R} \), there exists a unique Radon measure \( \mu \) on the Borel \(\sigma\)-algebra of \( X \) such that
    \[
    L(f) = \int_X f \, d\mu, \quad \forall f \in C(X).
    \]

    \end{theorem}
    A point \( p \) in a convex set \( S \) is called an \emph{extreme point} if \( p \) cannot be expressed as a non-trivial convex combination of other points in \( S \). Formally, \( p \in S \) is an extreme point if, whenever \( p = t x + (1 - t) y \) for some \( x, y \in S \) and \( t \in (0,1) \), it follows that \( x = p \) and \( y = p \).
    \begin{theorem}[Krein-Milman]\label{kreinmilman}
        Let \( X \) be a compact convex subset of a locally convex topological vector space. Then \( X \) is the closed convex hull of its extreme points, i.e.,
        \[
            X = \overline{\operatorname{conv}}(\operatorname{ext}(X)),
        \]
        where \( \operatorname{ext}(X) \) denotes the set of extreme points of \( X \), and \( \overline{\operatorname{conv}} \) denotes the closure of the convex hull.
    \end{theorem}
    
    \begin{theorem}[Phelps \cite{phelps1963extreme}]\label{phelps}
        Let $A$ and $B$ be algebras of bounded real valued functions on the sets $X$ and $Y$ respectively, each containing the constant functions. Let $K(A,B)$ be the convex set of all linear operators $T$ from $A$ to $B$ which satisfy $T\geq 0$ and $T(1)=1$. Then $T$ is an extreme point of $K(A,B)$ if and only if $T$ is multiplicative.
    \end{theorem}

\section{C*-algebras}\label{cstar}

In this section we recall some basic definitions and results from the theory of C*-algebras that we will use in the paper. For a more detailed exposition the reader may check \cite{farah2016model}.

A \emph{C*-algebra} is a complex algebra \( A \) equipped with a norm \( \| \cdot \| \) and an involution \( *: A \to A \) such that:
\begin{enumerate}
    \item \( A \) is a Banach space with respect to the norm \( \| \cdot \| \).
    \item The involution satisfies \( \| a^* a \| = \| a \|^2 \) for all \( a \in A \).
    \item The C*-identity holds: \( \| ab \| \leq \| a \| \cdot \| b \| \) for all \( a, b \in A \).
\end{enumerate}

By the Gelfand-Naimark theorem, every C*-algebra is isometrically isomorphic to a norm-closed self-adjoint subalgebra of \( B(H) \), the algebra of bounded linear operators on a Hilbert space \( H \).

The unitary group of a C*-algebra \( A \), denoted \( U(A) \), is the set of all invertible elements in \( A \) with respect to the multiplication operation.

There are several topologies on \( U(A) \). We are only interested in the following two:
\begin{enumerate}
    \item The \emph{norm topology}, which is the topology induced by the norm of \( A \).
    \item The \emph{relative Banach-space weak topology}, which is the topology induced by the weak topology on the dual of $A$ as a Banach space. 
\end{enumerate}

The (relative) weak topology is coarser than the norm topology, and \( U(A) \) is a topological group with respect to both topologies.

A map $\varphi: A \to B$ is completely positive if
$\varphi \otimes \text{id}_n : A \otimes M_n(\mathbb{C}) \to B \otimes M_n(\mathbb{C})$
is positive for all $n$. The acronym c.p.c. stands for completely positive and contractive,
where contractive means 1-Lipschitz.

A C*-algebra \( A \) is said to be \emph{nuclear} if for every tuple \(\bar{a}\) in the unit ball \(A_1\) and every \(\epsilon > 0\), there are a finite-dimensional C*-algebra \(F\) and completely positive contractive (c.p.c.) maps \(\varphi: A \to F\) and \(\psi: F \to A\) such that the diagram
\[
\begin{tikzcd}
A \arrow[r, "\varphi"] \arrow[rr, "id", bend left] & F \arrow[r, "\psi"] & A
\end{tikzcd}
\]
\(\epsilon\)-commutes on \(\bar{a}\).

\begin{theorem}[Paterson \cite{paterson1992nuclear}]\label{paterson}
    Let \( A \) be a C*-algebra. The unitary group \( U(A) \) is amenable with respect to the relative Banach-space weak topology if and only if \( A \) is nuclear.
\end{theorem}

\begin{theorem}[\cite{ge2001ultraproducts} Proposition 2.1]\label{Ge}
    The ultraproduct of unital $C^*$-algebras is a unital $C^*$-algebra and the unitary group of the ultraproduct is the ultraproduct of the unitary groups.
\end{theorem}

\bibliographystyle{alpha}
\bibliography{bibliog}

\end{document}